\documentclass[12pt,reqno]{amsart}

\usepackage[a4paper, margin=3cm]{geometry}

\usepackage{amssymb} 
\usepackage{tikz-cd}
\usepackage[T1]{fontenc}

\usepackage[colorlinks=true, allcolors=blue]{hyperref}

\usepackage{enumitem}
\setlist[enumerate,1]{label=\arabic*.,leftmargin=*}
\newlist{romanenum}{enumerate}{1}
\setlist[romanenum,1]{label=(\roman*),leftmargin=*}

\numberwithin{table}{section}

\numberwithin{equation}{section}

\newtheorem{theorem}{Theorem}[section]
\newtheorem{lemma}[theorem]{Lemma}%[section]
\newtheorem{corollary}[theorem]{Corollary}%[section]
\theoremstyle{definition}
\theoremstyle{remark}
\newtheorem{remark}[theorem]{Remark}%[section]

\newcommand{\N}{\mathbb{N}}
\newcommand{\Z}{\mathbb{Z}}
\newcommand{\Q}{\mathbb{Q}}
\newcommand{\R}{\mathbb{R}}
\newcommand{\C}{\mathbb{C}}
\newcommand{\F}{\mathbb{F}}

\DeclareMathOperator{\Ext}{Ext}
\DeclareMathOperator{\GL}{GL}
\DeclareMathOperator{\Hom}{Hom}

\DeclareMathOperator{\img}{im}
\DeclareMathOperator{\rk}{rk}
\DeclareMathOperator{\vspan}{span}

\newcommand{\Alt}{\mathchoice{{\textstyle\bigwedge}}%
    {{\bigwedge}}%
    {{\textstyle\wedge}}%
    {{\scriptstyle\wedge}}}

\newcommand\cspin{$^c$}
\newcommand{\spinc}{\text{spin\cspin} }%
\newcommand{\Spinc}{\text{Spin\cspin} }%

\DeclareMathOperator{\Spin}{Spin}

\usepackage{newfloat}
\DeclareFloatingEnvironment[
    fileext=lod,
    name=Diagram,
    placement=hp,
    within=section]{fdiagram}

\begin{document}

% NOTE: \large\cspin in the title - only for amsart look & feel
\title{\texorpdfstring{Spin{\cspin}}{Spinc} structures on real Bott manifolds}

\author[A.~G\k{a}sior]{Anna G\k{a}sior}
\address{Maria Curie-Sk{\l}odowska University, Institute of Mathematics\\
pl. Marii Curie-Sk{\l}odowskiej 1, 20-031 Lublin, Poland}
\email{anna.gasior@umcs.pl}

\author[R.~Lutowski]{Rafa{\l} Lutowski}
\address{Institute of Mathematics, Faculty of Mathematics, Physics and Informatics\\ University of Gda\'nsk\\
Wita Stwosza 57, 80-308 Gda\'nsk, Poland}
\email{rafal.lutowski@ug.edu.pl}

\date{July 1, 2024}

\begin{abstract}
We give a necessary and sufficient condition for existence of \spinc structures on real Bott manifolds.
\end{abstract}

\maketitle

\section{Introduction}

Let $\Gamma$ be a fundamental group of an $n$-dimensional \emph{real Bott manifold} $M$. From \cite{KM09} we know that $\Gamma$ is a Bieberbach group (and $M$ is a flat manifold) of \emph{diagonal type}. This means that:
\begin{romanenum}[widest=ii]
\item $\Gamma$ is torsion-free and fits into a short exact sequence
\begin{equation}
\label{eq:ses}
0\longrightarrow \Z^n\stackrel{}{\longrightarrow} \Gamma \stackrel{\pi}{\longrightarrow} C_2^d \longrightarrow 1,
\end{equation}
for some natural number $d$, where $C_2$ is the cyclic group of order $2$;
\item there exists a basis of $\Z^n$, such that the image of $\varrho \colon C_2^d\to \GL(n,\Z)$ is a group of diagonal matrices, where
\begin{equation}
\label{eq:conjugation}
\varrho_g(z)=\gamma z \gamma^{-1},
\end{equation}
for $z \in \Z^n, g\in C_2^d$ and $\gamma\in\Gamma$ such that $\pi(\gamma)=g$.
\end{romanenum}
Note that in the case of real Bott manifolds $d=n$.

In this note, we examine the existence of \spinc structures on orientable real Bott manifolds. 
Recall, that for $n \geq 3$, $\Spin(n)$ is the universal (and a double) cover of $SO(n)$, and
\[
\Spin^c(n) := \Spin(n) \times S^1 / \langle (-1,-1) \rangle = \Spin(n) \times_{C_2} S^1,
\]
where $C_2 = \{ \pm 1 \}$ is the center of $\Spin(n)$. \Spinc structure on $M$ is defined as an equivariant lift of its tangent bundle, regarded as a $SO(n)$-principal bundle. This means, in particular, that there exists a principal $\Spin^c(n)$ bundle $P$ and a map $\Lambda \colon P \to TM$, such that the following diagram commutes:
\begin{center}
\begin{tikzpicture}
\node (x) at (2,0) {$M$};
\node (q) at (0,-1) {$TM$};
\node (p) at (0,1) {$P$};
\node (sp) at (-3,1) {$P \times \Spinc(n)$};
\node (so) at (-3,-1) {$TM \times SO(n)$};
\draw[->] (sp) -- (p);
\draw[->] (sp) --node[right]{$\Lambda \times \lambda$} (so);
\draw[->] (so) -- (q);
\draw[->] (p) --node[left]{$\Lambda$} (q);
\draw[->] (p) -- (x);
\draw[->] (q) -- (x);
\end{tikzpicture}
\end{center}

Note that $\lambda$ is the unique map defined by the following commutative diagram
\[
\begin{tikzpicture}
\node[anchor=east] (product) at (-5em,2em) {$\Spin(n) \times S^1$};
\node (spin) at (0em,2em) {$\Spin(n)$}; 
\node[anchor=west] (so) at (5em,2em) {$SO(n)$};
\node[anchor=north] (spinc) at (0em,-1em) {$\Spinc(n)$};
\draw[->] (product) --node[above, inner sep=.3em] {\footnotesize$p$} (spin);
\draw[->] (spin) --node[above, inner sep=.3em] {\footnotesize$l$} (so);
\draw[->] (product) --node[below left, inner sep=.3em] {\footnotesize$\nu$} (spinc);
\draw[->, dashed] (spinc) --node[below right, inner sep=.3em] {\footnotesize$\lambda$} (so);
\end{tikzpicture}
\]
where $p$ is the projection to the first coordinate, $l$ is the covering map and $\nu$ is the natural homomorphism.

\Spinc structures are, in a natural way, of the interest in the case of those manifolds, which do not admit any spin structure. 
Spin structures on real Bott manifolds were considered in \cite{G17} and \cite{D18}. As extensions of these results, the problem of existence of spin structures on generalized real Bott manifolds was considered in \cite{DU19} and \cite{GGP23}, and the case of flat manifolds of diagonal type -- in \cite{GS14} and \cite{LPPS19}. Up to now, the problem of existence of \spinc structures for the flat manifolds of diagonal type has been considered only in the case of Hantzsche-Wendt manifolds in \cite{LPS22}. In this paper, we extend the results to the family of real Bott manifolds.

Every real Bott manifold $M$ is determined -- up to diffeomorphism -- by a certain strictly upper triangular matrix $A$ with coefficients in $\F_2$.
We call $A$ a \emph{Bott matrix}, and we denote the manifold $M$ by $M(A)$. Since $M$ has so nice combinatorial description, it is natural to ask if it is possible to give a similar condition, under which $M$ admits a \spinc structure. Our final Theorem \ref{theorem:spinc_combinatorial} presents this kind of approach.

The structure of the article is as follows. Section \ref{sec:spinc} describes general facts for cohomology groups of a flat manifold $M$, which are crucial in determining whether $M$ admits a \spinc structure. Section \ref{sec:rbm} restricts those considerations to the specific case, when $M$ is a real Bott manifold. In particular, we get a description of the image of the map $H^2(M,\Z) \to H^2(M,\F_2)$ induced by reduction coefficients modulo $2$. Finally, Section \ref{sec:rbm_spinc} gathers previous results to give necessary and sufficient conditions for the existence of a \spinc structure on $M$. These are then used in Section \ref{sec:numbers} to show some results on number of \spinc real Bott manifolds in low dimensions.

\section{\texorpdfstring{\Spinc}{Spinc} structures on flat manifolds}
\label{sec:spinc}

This section contains some preliminary considerations on the existence of \spinc structures on an orientable $n$-dimensional flat manifold $M$. Let $\Gamma = \pi_1(M)$. Since $M$ is an Eilenberg-MacLane space of type $K(\Gamma,1)$, cohomology groups of $M$ and $\Gamma$ coincide, see \cite[Section 2.1]{Luc10} and \cite[Theorem I]{EMcL45}.

For every $i \in \N$, let 
\[
\beta^{(i)} \colon H^i(\Gamma,\F_2) \to H^{i+1}(\Gamma,\F_2) \quad\text{and}\quad \tilde \beta^{(i)} \colon H^i(\Gamma,\F_2) \to H^{i+1}(\Gamma,\Z)
\]
be the Bockstein maps corresponding to the extensions
\begin{equation}
\label{eq:bockstein4}
0 \longrightarrow \F_2 \longrightarrow \Z/4 \longrightarrow \F_2 \longrightarrow 0
\end{equation}
and
\begin{equation*}
% \label{eq:bocksteinZ}
0 \longrightarrow \Z \stackrel{\iota}{\longrightarrow} \Z \stackrel{\rho}{\longrightarrow} \F_2 \longrightarrow 0
\end{equation*}
respectively. Note that $\iota$ and $\rho$ are given by multiplication by $2$ and reduction modulo $2$, respectively. By \cite[Chapter 3.E]{H02}, we have commutative Diagram \ref{diagram:bockstein}, with horizontal and vertical sequences being exact. Note that $\iota^{(i)}$ and $\rho^{(i)}$ are the homomorphisms induced by $\iota$ and $\rho$, respectively. By \cite[page 49]{F00}, $M$ admits a \spinc structure if and only if $w_2(M) \in \img \rho^{(2)}$, where $w_2(M)$ is the second Stiefel-Whitney class of $M$. Using   Diagram \ref{diagram:bockstein} we immediately get
\begin{corollary} \ 
\begin{enumerate}
\item If $w_2(M) \in \img \beta^{(1)}$, then $M$ admits a \spinc structure.
\item If $w_2(M) \not\in \ker \beta^{(2)}$, then $M$ does not admit any \spinc structure.
\end{enumerate}
\end{corollary}

\begin{fdiagram}[t]
\[
\begin{tikzcd}[sep=large]
& & H^2(\Gamma,\Z) \arrow{d}{\iota^{(2)}}\\
H^1(\Gamma,\Z) \arrow{r}{\rho^{(1)}} & H^1(\Gamma,\F_2)  \arrow{r}{\tilde{\beta}^{(1)}} \arrow{dr}{\beta^{(1)}} & H^2(\Gamma,\Z) \arrow{d}{\rho^{(2)}} \\
& & H^2(\Gamma,\F_2) \arrow{ld}[above]{\beta^{(2)}} \arrow{d}{\tilde\beta^{(2)}}\\
& H^3(\Gamma, \F_2) & H^3(\Gamma, \Z) \arrow{l}[above]{\rho^{(3)}}
\end{tikzcd}
\]
\caption{Diagram of Bockstein homomorphisms}
\label{diagram:bockstein}
\end{fdiagram}

For an abelian group $K$, the universal coefficient theorem gives the following formula:
\begin{equation}
\label{eq:universal-coeff-thm}
H^i(\Gamma, K) \cong \Hom_{\Z}(H_i(\Gamma),K) \oplus \Ext^1_{\Z}(H_{i-1}(\Gamma), K)
\end{equation}
(see \cite[Corollary 7.60]{Ro09}).

By Bieberbach theorems (see \cite[Theorem 2.1]{Sz12}), $\Gamma$ fits into the following short exact sequence
\begin{equation}
\label{eq:bieberbach}
0 \longrightarrow L \longrightarrow \Gamma \longrightarrow G \longrightarrow 1,
\end{equation}
where $G$ is finite, $L$ is a free abelian group of rank $n$, with a $G$ module structure given by conjugations in $\Gamma$, similarly as in \eqref{eq:conjugation}. Moreover, $L$ is the unique maximal abelian normal subgroup of $\Gamma$. We call $G$ the \emph{holonomy group} of $\Gamma$. The before-mentioned action of $G$ on $L$ corresponds to the \emph{holonomy representation} of $\Gamma$.

For a finitely generated abelian group $K$, by $\rk_{\Z}K$ we denote the rank of its torsion-free part.

\begin{lemma} 
\label{lemma:hi_z_rank}
Let $i \in \N$. Then,
\[
\rk_{\Z} H^i(\Gamma, \Z) = \dim_{\Q} \left(\Alt^i (L^* \otimes_{\Z} \Q)\right)^G.
\]
\end{lemma}
\begin{proof}
Using formula \eqref{eq:universal-coeff-thm}, we get that the torsion-free part of $H^i(\Gamma,\Z)$ is exactly $\Hom_{\Z}(H_i(\Gamma),\Z)$. Its rank equals the dimension of $H^i(\Gamma, \Q) = \Hom_{\Z}(H_i(\Gamma),\Q)$. By \cite[Lemma 1.2]{Hi85}
\[
H^i(\Gamma, \Q) \cong \left(\Alt^i (L^* \otimes_{\Z} \Q)\right)^G.
\]
\end{proof}

\begin{remark}
Let $i \in \N$. The $\Alt^i$ denotes the $i$-th exterior power and for any $G$-module $W$
\[
W^G = \left\{ w \in W : \forall_{g \in G} gw = w \right\}.
\]
Note that $\rk_{\Z}H^i(\Gamma,\Z) = \rk_{\Z}H_i(\Gamma) = \rk_{\Z}H_i(M)$ is the $i$-th Betti number of $M$, which we will denote by $b_i(M)$ or simply $b_i$, if the underlying manifold can be read out of the context. In particular, we have
\[
b_1 = \dim_{\Q} \left(\Alt^1 (L^* \otimes_{\Z} \Q)\right)^G = \dim_{\Q} (L^* \otimes_{\Z} \Q)^G = \dim_{\Q} (L \otimes_{\Z} \Q)^G = \rk_{\Z}L^G
\]
(see \cite[Corollary 1.3]{HS86}).
\end{remark}

\begin{remark}
For an abelian group $K$, we have 
\[
H^1(\Gamma,K) = \Hom(\Gamma,K).
\]
Every homomorphism from the above group must factor by the abelianization of $\Gamma$, which equals $H_1(\Gamma)$ (see \cite[Theorem 2A.1]{H02}). Moreover, as $M$ is a connected manifold, $H_0(M) = H_0(\Gamma) = \Z$. Hence, $\Ext_{\Z}(H_0(\Gamma),K) = 0$ and the above agrees with the formula
\[
H^1(\Gamma,K) = \Hom_{\Z}(H_1(\Gamma),K).
\]
\end{remark}

\section{The diagram for real Bott manifolds}
\label{sec:rbm}

In this section, we describe most of the groups and homomorphisms, which appear in Diagram \ref{diagram:bockstein} in the case of real Bott manifolds.

Let $A = [a_{ij}] \in \F^{n \times n}$ be a strictly upper triangular matrix. The real Bott manifold $M := M(A)$ is a quotient of a flat torus $T^n = (S^1)^n$ by the free action of an elementary abelian group $C_2^n = \langle c_1, \ldots, c_n \rangle$. Identifying $S^1$ with the unit circle in $\C$ the action of $c_i$ on $(z_1,\ldots,z_n) \in T^n$ is given by
\[
c_i \cdot (z_1,\ldots,z_n) = (z_1, \ldots, z_{i-1}, -z_i, c_{i,i+1}(z_{i+1}), \ldots, c_{i,n}(z_n)),
\]
where
\[
c_{i,j}(z) = \left\{
\begin{array}{ll}
z & \text{ if } a_{ij} = 0\\
\overline{z} & \text{ if } a_{ij} = 1
\end{array}
\right.
\]
for every $1 \leq i < j \leq n$.

Assume that $\Gamma = \pi_1(M)$ fits into the short exact sequence \eqref{eq:ses}, where $d=n$, and that it is generated by $s_1, \ldots, s_n$, where 
\begin{equation}
\label{eq:generators_of_gamma}
\pi(s_i) = c_i
\end{equation}
for $1 \leq i \leq n$. 

\begin{remark}
\label{remark:generators}
To be more precise, we think of $\Gamma$ as a subgroup of $\GL(n,\R) \ltimes \R^n$ and we take $s_i=(A,a)$ where
\[
A = \operatorname{diag}(1,\ldots,1,(-1)^{a_{i,i+1}}, \ldots, (-1)^{a_{i,n}})
\]
and
\[
a = \bigg(\underbrace{0,\ldots,0}_{i-1},\frac{1}{2},0,\ldots,0\bigg),
\]
when treating $A$ as an integer matrix (see \cite[Proposition 1.6]{Sz12} and \cite[Formula (3-1)]{KM09}).
\end{remark}

\subsection{Cohomology over the integers}

\begin{lemma}
\label{lemma:h1_gamma}
$H_1(\Gamma) = \F_2^{n-b_1} \oplus \Z^{b_1}$.
\end{lemma}
\begin{proof}
$H_1(\Gamma)$ is the abelianization of $\Gamma$. By \cite[Lemma 3.2]{KM09} the generators $s_1,\ldots,s_n$ can be taken in such a way, that for every $1 \leq i < j \leq n$ we have
\[
[s_j, s_i] = 
\left\{
\begin{array}{ll}
s_j^{-2} & \text{ if } a_{ij} = 1\\
1 & \text{ if } a_{ij} = 0
\end{array}
\right.
\]
Hence
\[
\langle s_j[\Gamma, \Gamma] \rangle = \left\{
\begin{array}{ll}
\Z & \text{ if } A^{(j)} = 0\\
\F_2 & \text{ if } A^{(j)} \neq 0
\end{array}
\right.
\]
where $A^{(j)}$ is the $j$-th column of $A$, for every $1 \leq j \leq n$. Since in the short exact sequence \eqref{eq:ses} $\Z^n$ is generated by the squares of the generators of $\Gamma$, we easily get that zero columns of $A$ correspond to trivial constituents of the action of $C_2^n$ on $\Z^n$, denoted by $\varrho$ in \eqref{eq:conjugation}. By definition, the action of $\img(\varrho)$ on $\Z^n$ is isomorphic over the rationals to the holonomy representation. Hence, assuming that $\Gamma$ fits into the short exact sequence \eqref{eq:bieberbach}, we have
\[
b_1 = \rk_{\Z}L^G = \rk_{\Z}(\Z^n)^{\img\varrho} = \rk_{\Z}(\Z^n)^{C_2^n}.
\]
\end{proof}

From the above proof, we immediately get

\begin{corollary} 
\label{corollary:b1}
$H^1(\Gamma, \Z) = \Z^{b_1}$. Moreover, $b_1$ is the number of zero columns of $A$.
\end{corollary}

\begin{corollary}
\label{corollary:h2_gamma_z}
$H^2(\Gamma, \Z) = \F_2^{n-b_1}\oplus\Z^{b_2}$. Moreover, the second Betti number $b_2$ is equal to the number of pairs of equal columns of $A$:
\begin{equation}
\label{eq:b2}
b_2 = \#\left\{(i,j): 1 \leq i < j \leq n \text{ and } A^{(i)} = A^{(j)}\right\}.
\end{equation}
\end{corollary}
\begin{proof}
Let $e_1,\ldots,e_n$ be the basis of $\Z^n$, where $e_j := s_j^2$ for $1 \leq j \leq n$. If one treats $A$ as an integer matrix, as in Remark \ref{remark:generators}, then
\[
\forall_{1 \leq i,j \leq n} \varrho(c_i)(e_j) = (-1)^{a_{ij}} e_j.
\]
Since the action of $C_2^n$ is diagonal,  modules $\Z^n$ and $(\Z^n)^*$ are isomorphic and moreover, a basis of $\left(\Alt^2 \Z^n\right)^{C_2^n}$ is the set
\[
\left\{ e_i \wedge e_j : i < j \text{ and } A^{(i)} = A^{(j)} \right\}.
\]
Hence, by Lemma \ref{lemma:hi_z_rank}, we get equality \eqref{eq:b2}. By Lemma \ref{lemma:h1_gamma}, $\Ext_{\Z}^1(H_1(\Gamma),\Z)$ is an elementary abelian $2$-group of rank $n-b_1$, which ends the proof.
\end{proof}

\begin{lemma}
\label{lemma:dimension-of-img-rho}
$\dim_{\F_2} \img \rho^{(2)} = n-b_1 + b_2$.
\end{lemma}
\begin{proof}
By the isomorphism theorem and using exactness of the column in Diagram \ref{diagram:bockstein}, we have 
\[
\img \rho^{(2)} \cong H^2(\Gamma,\Z)/\ker \rho^{(2)} = H^2(\Gamma,\Z)/\img \iota^{(2)}.
\]
Since $\iota^{(2)}$ is induced by the multiplication by $2$, using Corollary \ref{corollary:h2_gamma_z} we get that $\img\iota^{(2)} = 2\Z^{b_2}$,
\[
H^2(\Gamma,\Z)/\img \iota^{(2)} \cong \F_2^{n-b_1}\oplus\Z^{b_2}/2\Z^{b_2} \cong \F_2^{n-b_1+b_2},
\]
and the result follows.
\end{proof}

\subsection{Cohomology over \texorpdfstring{$\F_2$}{F2}}

By Lemma \ref{lemma:h1_gamma}, we easily get that $H^1(\Gamma, \F_2) = \F_2^n$. It is desirable, however, to describe this group in more detail, since it generates the cohomology ring of $\Gamma$.

Assume that $H^1(C_2^n, \F_2) = \Hom(C_2^n, \F_2)$ is generated by $c_1^*, \ldots, c_n^*$, which are in a way elements dual to the generators of $C_2^n$, i.e. $c_i^*(c_j) = \delta_{ij}$, where $\delta_{ij}$ is the Kronecker delta, for $1 \leq i,j \leq n$. Define
\[
x_i := \pi_*(c_i^*) \in H^1(\Gamma, \F_2)
\]
for $1 \leq i \leq n$. Since generators $s_1,\ldots,s_n$ satisfy \eqref{eq:generators_of_gamma}, we have
\begin{equation}
\label{eq:h1_on_generators}
\forall_{1 \leq i,j \leq n} x_i(s_j) = \delta_{ij}.
\end{equation}

\begin{lemma}[{\cite[Lemma 2.1]{KM09}}]
\[
H^*(\Gamma, \F_2) = \langle x_1,\ldots,x_n \rangle \cong \F_2[x_1,\ldots,x_n]/I_A,
\]
where $I_A$ is the ideal generated by the polynomials
\begin{equation*}
%\label{eq:theta}
\theta_j := x_j^2 + x_j\sum_{i=1}^{j-1}a_{ij}x_i,
\end{equation*}
for $1 \leq j \leq n$.
\end{lemma}

\begin{remark}
\label{remark:relations}
In what follows, we will treat $H^*(\Gamma, \F_2)$ not as a quotient, but rather as a ring which satisfies relations given by the generators of the ideal $I_A$.
\end{remark}

For every $1 \leq j \leq n$, define
\begin{equation}
\label{eq:alpha_j}
\alpha_j := \sum_{i=1}^{j-1}a_{ij}x_i = \sum_{i=1}^{n}a_{ij}x_i.
\end{equation}
Using the above remark, we can write the relations of $H^*(\Gamma, \F_2)$ as
\begin{equation}
\label{eq:squares_relations}
x_1^2 = \alpha_1x_1, x_2^2 = \alpha_2x_2, \ldots, x_n^2 = \alpha_nx_n.
\end{equation}
\begin{corollary}
\label{corollary:squares}
$x_j^2 = 0$ if and only if $A^{(j)} = 0$, for every $1 \leq j \leq n$.
\end{corollary}

We are interested in the calculation of the kernel of $\beta^{(2)}$. First, we provide a formula for this map, using the fact that 
\begin{equation}
\label{eq:basis_of_ring}
\mathcal B_2 := \{x_ix_j : 1 \leq i < j \leq n\}
\end{equation}
is a basis of $H^2(\Gamma, \F_2)$ (see \cite[proof of Lemma 2.1]{G17}).

\begin{lemma}
\label{lemma:bockstein2}
Let $1 \leq i,j \leq n$. Then
\[
\beta^{(2)}(x_ix_j) = x_i^2x_j + x_ix_j^2.
\]
\end{lemma}
\begin{proof}
Assume, that $1 \leq i,j \leq n$.
Note that in the calculation of the Bockstein homomorphism, we use short exact sequence \eqref{eq:bockstein4}. Define two maps:
\begin{enumerate}
\item $\overline{\phantom{0}} \colon \F_2 \to \Z/4$ -- the lift, given by $\overline 0 = 0, \overline 1 = 1$.
\item $/2 \colon 2\Z/4 \to \F_2$ -- a division, given by $2/2 = 1, 0/2 = 0$.
\end{enumerate}
For the simplicity of calculations, we abuse the notation and identify a cohomology class with a specific cocycle which defines it. With this in hand, $\beta^{(2)}$ is given by
\begin{equation}
\label{eq:beta2}
\begin{split}
\beta^{(2)}(x_ix_j)(a,b,c) &= \delta^2(\overline{x_ix_j})(a,b,c)/2\\
&= \bigl(\overline{x_ix_j}(b,c) - \overline{x_ix_j}(ab,c) + \overline{x_ix_j}(a,bc) - \overline{x_ix_j}(a,b)\bigr)/2
\end{split}
\end{equation}
for $a,b,c \in \{s_1, \ldots, s_n\}$, since polynomials in $H^*(\Gamma, \F_2)$ are recognized by their values on the generators of $\Gamma$, as in \eqref{eq:h1_on_generators}. We have
\[
\overline{x_ix_j}(ab,c) = \overline{x_i(ab)} \cdot \overline{x_j(c)} = \left(\overline{x_i(a)+x_i(b)}\right) \cdot \overline{x_j(c)}.
\]
Using similar calculations for the rest of the summands in \eqref{eq:beta2} we get, that its value depends only on $x_i(a), x_i(b), x_j(b), x_j(c)$. Considering all $16$ cases gives us, that
\[
\beta^{(2)}(x_ix_j)(a,b,c) = 1 \Leftrightarrow \bigl(x_i(a) = x_j(c) = 1 \;\wedge\; x_i(b) \neq x_j(b)\bigr).
\]
This occurs exactly when $i \neq j, a=s_i, c=s_j$ and $b \in \{s_i,s_j\}$. Since, in addition, $\beta^{(2)}(x_j^2) = 0$, hence we get the general formula
\[
\beta^{(2)}(x_ix_j) = x_i^2x_j + x_ix_j^2.
\]
\end{proof}

\begin{corollary}
\label{corollary:squares_in_kernel}
Let $S := \vspan\{ x_j^2 : 1 \leq j \leq n\}$. Then $\dim S = n-b_1$ and $S \subset \ker \beta^{(2)}$.
\end{corollary}
\begin{proof}
$S$ is in the kernel of $\beta^{(2)}$ by Lemma \ref{lemma:bockstein2}. Now, by Corollaries \ref{corollary:b1} and \ref{corollary:squares}, the generating set $\{ x_j^2 : 1 \leq j \leq n\}$ has $n-b_1$ non-zero elements. Since $\mathcal B_2$, defined by \eqref{eq:basis_of_ring}, is a basis of $H^2(\Gamma, \F_2)$, using relations \eqref{eq:squares_relations} we get that these $n-b_1$ non-zero elements are linearly independent.
\end{proof}

\begin{theorem}
\label{theorem:image}
Let
\[
S_1 := \{\alpha_jx_j: 1 \leq j \leq n, A^{(j)} \neq 0\} \text{ and } S_2:=\{x_kx_l : 1 \leq k < l \leq n, A^{(k)} = A^{(l)}\}.
\]
Then $S_1 \cup S_2$ is a basis of $\img\rho^{(2)}$.
\end{theorem}
\begin{proof}
By Diagram \ref{diagram:bockstein} we have that $\img \rho^{(2)} = \ker \tilde \beta^{(2)} \subset \ker \beta^{(2)}$. By Lemma \ref{lemma:dimension-of-img-rho} it is enough to show that $S_1 \cup S_2$ generates $\ker \beta^{(2)}$.

First note that if for some $k < l$,  $A^{(k)} = A^{(l)}$, then $\alpha_k = \alpha_l$ and
\begin{equation}
\label{eq:xkxl}
\beta^{(2)}(x_kx_l) = x_k^2x_l + x_kx_l^2 = \alpha_k x_kx_l + x_k \alpha_l x_l = 0.
\end{equation}
Together with Corollary \ref{corollary:squares_in_kernel} we get that $\vspan S_1 \cup S_2 \subset \ker \beta^{(2)}$.

Assume that $\vspan S_1 \cup S_2 \neq \ker \beta^{(2)}$. Then there exists an element 
\[
f = f_1x_1 + \ldots + f_lx_l \in \ker \beta^{(2)},
\]
such that $f_i \in \F_2[x_1,\ldots,x_{i-1}]\setminus\{\alpha_i\}$ for $1 \leq i \leq l$ and $f_l \neq 0$. Let
\[
f_l = \sum_{k=1}^{l-1} f_{kl} x_k,
\]
where $f_{1l},\ldots,f_{l-1,l} \in \F_2$. Using \eqref{eq:xkxl}, we can assume that 
\begin{equation}
\label{eq:kl0}
\forall_{k < l} A^{(k)} = A^{(l)} \Rightarrow f_{kl} = 0.
\end{equation}
We have
\[
\beta^{(2)}(f) = f_1^2x_1 + f_1x_1^2 + \ldots f_l^2x_l+f_lx_l^2 = 0.
\]
Because the only relations in $H^*(\Gamma,\F_2)$ involve squares, we get that
\[
\mathcal B_3 := \{x_ix_jx_k : i < j < k\}
\]
is a basis of $H^3(\Gamma,\F_2)$. Hence, 
\[
0 = f_l^2x_l+f_lx_l^2 = f_l^2x_l+f_l\alpha_lx_l = (f_l^2+f_l\alpha_l)x_l.
\]
Because $\mathcal B_3$ is a basis and $f_l^2+f_l\alpha_l \in \F_2[x_1,\ldots,x_{l-1}]$, we get that $f_l(f_l+\alpha_l) = 0$. By assumption, $f_l$ and $f_l+\alpha_l$ are non-zero polynomials. Hence, $\theta_k$ is a summand of $f_l(f_l+\alpha_l)$ for some $k$. Choose maximal such $k$. This means, that $f_l \in \F_2[x_1,\ldots,x_k]$ is such that 
\begin{equation}
\label{eq:kl1}
f_{kl} = 1
\end{equation}
and $\alpha_l \in \F_2[x_1,\ldots,x_{k-1}]$. We get
\[
\theta_k = x_k^2 + \sum_{i < k} f_{il}x_ix_k + \sum_{i < k} (f_{il}+a_{il})x_ix_k = x_k^2 + \left(\sum_{i < k} a_{il}x_i \right)x_k = x_k^2 + \alpha_lx_k,
\]
since $a_{il} = 0$ for $i \geq k$. Hence $\alpha_k x_k = \alpha_l x_k$. Both sides of the equality are sums of basis elements, thus $\alpha_k = \alpha_l$ and $A^{(k)} = A^{(l)}$. By \eqref{eq:kl0} and \eqref{eq:kl1} we get a contradiction, and the result follows.
\end{proof}

\section{\texorpdfstring{\Spinc}{Spinc} structures on real Bott manifolds}
\label{sec:rbm_spinc}

We keep the notation of previous sections. Let $w_2:=w_2(M)$ be the second Stiefel-Whitney class of an orientable real Bott manifold $M$, defined by a matrix $A \in \F_2^{n \times n}$. Recall, that $M$ admits a \spinc structure if and only if $w_2 \in \img \rho^{(2)}$. Our goal in this section is to describe this condition in terms of certain properties of the matrix $A$. Recall, that we think of $H^*(\Gamma, \F_2)$ as a ring with relations (see Remark \ref{remark:relations}), so that $w_2 \in F_2[x_1,\ldots,x_n]$. Let
\[
w_2 = \sum_{i \leq j} \alpha_{ij} x_ix_j
\]
and define a new polynomial $w_2' = \sum \alpha_{ij}' x_ix_j$, where
\[
\alpha_{ij}' := 
\left\{
\begin{array}{ll}
0 & \text{ if } i=j \text{ or } A^{(i)} = A^{(j)}\\
\alpha_{ij} & \text{ otherwise}
\end{array}
\right.
\]
for all $i \leq j$. Let
\[
w_2' = \sum_{j=1}^n \alpha_j' x_j,
\]
where $\alpha_j' = \sum_{i=1}^{j-1} \alpha_{ij}'x_i$ for $1 \leq j \leq n$.
Note, that $w_2'$ comes from $w_2$ by eliminating monomials, which are squares and elements of $S_2$. By Corollary \ref{corollary:squares_in_kernel} and Theorem \ref{theorem:image} we thus get the following criterion for the existence of \spinc structure on $M$:

\begin{theorem}
\label{theorem:spinc}
$M$ admits a \spinc structure if and only if
\begin{equation}
\label{eq:spinc_condition_derivative}
\forall_{1 \leq j \leq n} \alpha_j' = 0 \vee \alpha_j' = \alpha_j.
\end{equation}
\end{theorem}

\begin{proof}
By the discussion preceding the theorem and Theorem \ref{theorem:image}, $M$ admits a \spinc structure if and only if $w_2' \in \vspan S_1$. This is clearly equivalent to the condition \eqref{eq:spinc_condition_derivative}.
\end{proof}

\begin{remark}
Note, that the above criterion does not depend on choosing representative of the second Stiefel-Whitney class of $M$. If $\hat w_2$ is another polynomial representing it, then $w_2 - \hat w_2 \in \vspan\{ \theta_1, \ldots, \theta_n \}$, hence $w_2' - \hat w_2' \in \vspan S_1$. Therefore, condition \eqref{eq:spinc_condition_derivative} is equivalently satisfied by $w_2$ and $\hat w_2$.
\end{remark}

% Our goal is to present the above criterion in terms of Bott matrix $A$. 
By \cite[Formula (2-4)]{KM09} one gets, that the total Stiefel-Whitney class of $M$ is equal to
\begin{equation}
\label{eq:total_sw_class}
w(M) = \prod_{j=1}^n(1+\alpha_j) \in H^*(\Gamma, \F_2).
\end{equation}
This in particular means that:
\begin{lemma}[{\cite[Lemma 2.2]{KM09}}]
\label{lemma:orientable}
$M$ is orientable if and only if sum in every row of $A$ is equal to $0$, i.e. every row of $A$ contains even numbers of $1$'s.
\end{lemma}

\begin{lemma}
\label{lemma:square_free}
% Let $M$ be an orientable real Bott manifold, defined by a matrix $A=[a_{ij}] \in \F^{n \times n}$. 
Let $w^{sf}_2$ denote the square-free part degree two component in the expansion of the total Stiefel-Whitney class \eqref{eq:total_sw_class}. Then
\[
w^{sf}_2 = \sum_{k < l} \langle A_{(k)}, A_{(l)} \rangle x_kx_l,
\]
where $A_{(k)}$ denotes the $k$-th row of $A$ and $\langle A_{(k)}, A_{(l)} \rangle$ is the scalar product, i.e.
\[
\langle A_{(k)}, A_{(l)} \rangle := \sum_{j=1}^n a_{kj}a_{lj},
\]
for $1 \leq k \leq l \leq n$.
\end{lemma}

\begin{proof}
Let $1 \leq k < l \leq n$. For $x,y \in \F_2$, let
\[
s_{xy} = \#\{ 1 \leq j \leq n : a_{kj} = x \text{ and } a_{lj} = y \}.
\]
Using formula \eqref{eq:alpha_j}, an easy calculation shows that the coefficient of $x_kx_l$ in $w^{sf}_2$ is equal to 
\[
s_{10}s_{01}+s_{10}s_{11}+s_{11}s_{01} \bmod 2.
\]
Since $M$ is orientable, by Lemma \ref{lemma:orientable} both $s_{10}+s_{11}$ and $s_{01}+s_{11}$ are even integers. Hence, in $\F_2$, we have $s_{10}=s_{11}=s_{01}$ and the above formula has the following form
\[
s_{10}s_{01}+s_{10}s_{11}+s_{11}s_{01} = 3 \cdot s_{11}^2 = s_{11} = \langle A_{(k)}, A_{(l)} \rangle.
\]
\end{proof}

Now, some summands of $w_2^{sf}$ may lie in $\vspan S_2$. If we rule them out of $w_2^{sf}$, we simply get $w_2'$. We have the following combinatorial criterion on the existence of \spinc structures on real Bott manifolds:
\begin{theorem}
\label{theorem:spinc_combinatorial}
Let $M$ be an orientable real Bott manifold, defined by a matrix $A \in \F_2^{n \times n}$. Let $A' = [a'_{ij}]\in \F_2^{n \times n}$ be defined as follows:
\[
\forall_{1 \leq i,j \leq n} \; a'_{ij} =
\left\{
\begin{array}{ll}
0 & \text{ if } i \leq j \text{ or } A^{(i)} = A^{(j)}\\
\langle A_{(i)}, A_{(j)} \rangle & \text{ otherwise } 
\end{array}
\right.
\]
Then $M$ admits a \spinc structure if and only if
\begin{equation}
\label{eq:spinc_cols}
A'^{(j)} = 0 \text{ or } A'^{(j)}  = A^{(j)}
\end{equation}
for every $3 \leq j \leq n-2$.
\end{theorem}

\begin{proof}
By Theorem \ref{theorem:spinc} and Lemma \ref{lemma:square_free} $M$ admits a \spinc structure if and only if \eqref{eq:spinc_cols} holds for every $1 \leq j \leq n$. We will show, that this condition always holds for $j \in \{1,2,n-1,n\}$. Let $A = [a_{ij}]$. By definition $A'^{(1)} = 0$ and if $A'^{(2)} \neq 0$, then $a'_{12} = 1$. Hence $0 = A^{(1)} \neq A^{(2)}$. Since $A$ is a strictly upper triangular matrix, we get that $a_{12} = 1$ and it is the only non-zero element in the second column of $A$. Therefore $A'^{(2)} = A^{(2)}$. Now, orientability of $M$ implies, that $A_{(n-1)} = A_{(n)} = 0$ (see Lemma \ref{lemma:orientable}). By definition of $A'$, we get that $A'^{(n-1)} = A'^{(n)} = 0$.
\end{proof}

It is well known, that every orientable four-dimensional manifold admits a \spinc structure (see \cite{KR85}). In the case of Theorem \ref{theorem:spinc_combinatorial}, this becomes trivial. However, we can use it in dimension $5$ to get
\begin{corollary}
Let $M$ be an orientable $5$-dimensional real Bott manifold, defined by a matrix $A=[a_{ij}]$. Then $M$ admits a \spinc structure if and only if
\begin{enumerate}[label=(\alph*), widest=a]
\item $A_{(3)} = 0$ or
\item $a_{12} = 0$ or \label{item:b}
\item $a_{23} = 0$. \label{item:c}
\end{enumerate}
\end{corollary}
\begin{proof}
We need to check condition \eqref{eq:spinc_cols} for $j=3$ only. If $A_{(3)} = 0$, then $A'^{(3)} = 0$. If $A_{(3)} \neq 0$, then $A_{(3)} = [0\,0\,0\,1\,1]$, since $M$ is orientable. Direct verification shows, that \eqref{eq:spinc_cols} holds if and only if submatrix
\[
\begin{bmatrix}
a_{12} & a_{13}\\
     0 & a_{23}
\end{bmatrix}
\]
of $A$ is in the one of the following six forms, which give exactly conditions \ref{item:b} and \ref{item:c}:
\[
\begin{bmatrix}
0 & 0\\
0 & 0
\end{bmatrix},
~\
\begin{bmatrix}
0 & 1\\
0 & 0
\end{bmatrix},
~\
\begin{bmatrix}
0 & 0\\
0 & 1
\end{bmatrix},
~\
\begin{bmatrix}
0 & 1\\
0 & 1
\end{bmatrix},
~\
\begin{bmatrix}
1 & 0\\
0 & 0
\end{bmatrix},
~\
\begin{bmatrix}
1 & 1\\
0 & 0
\end{bmatrix}.
\]
\end{proof}

\section{\texorpdfstring{\Spinc}{Spinc} structures in low dimensions}
\label{sec:numbers}

Based on Theorem \ref{theorem:spinc_combinatorial} we have built an application \cite{LData24}, which counts the number of real Bott manifolds with \spinc structures in dimensions up to 10. Note that in dimension $n$ there exist $\displaystyle 2^{\binom{n-1}{2}}$ of the orientable ones. In addition, with the usage of the criterion on existence of spin structures concerning vanishing second Stiefel-Whitney class, the number of spin real Bott manifolds has also been calculated. The results are shown in Table \ref{table:numbers}.

\begin{table}
\begin{center}
\def\arraystretch{1.5}
\begin{tabular}{l|c|c|c|c|c|c|c}
\textbf{dimension} & 4 & 5 & 6 & 7 & 8 & 9 & 10\\ \hline
\textbf{no. of \spinc RBM} & 8 & 52 & 592 & 7968 & 165712 & 4669464 & 191557024\\ \hline
\textbf{no. of spin RBM} & 6 & 24 & 72 & 672 & 1536 & 4416 & 181248 
\end{tabular}
\end{center}

\caption{Number of \spinc and spin real Bott manifolds (RBM) in low dimensions}
\label{table:numbers}
\end{table}

\bibliographystyle{plain}
\bibliography{bibl}

\end{document}